\RequirePackage{fix-cm}
\documentclass[smallextended]{svjour3}       
\smartqed  
\usepackage{algorithm}
\usepackage{algpseudocode}

\usepackage{amssymb}
\usepackage{amsmath}
\usepackage{epsf}

\newenvironment{tabcen}[1]{\begin{center}\begin{tabular}{#1}}{\end{tabular}\end{center}}

\title{Bent functions using Maiorana-McFarland secondary construction}
\titlerunning{Bent functions using Maiorana-McFarland secondary construction}  %

\author{Juan Carlos Ku-Cauich \and Javier Diaz-Vargas \and Sara Mandujano-Velazquez \and Victor Bautista-Ancona \thanks{The authors acknowledge the support of Mexican Conacyt}}
\authorrunning{Ku-Cauich, J. C., Diaz-Vargas J., Mandujano-Velazquez S., Bautista-Ancona V.}   
\institute{Corresponding author Juan Carlos Ku-Cauich \at
Computer Science, Cinvestav, Mexico City, Mexico, 
\email{jcku@cs.cinvestav.mx}
\and
Javier Diaz-Vargas \at
Facultad de Matem\'aticas, UADY, M\'erida Yucat\'an, Mexico, 
\email{javier.diaz@correo.uady.mx}
\and
Sara Mandujano-Velazquez \at
ESFM, IPN, Mexico City, Mexico, 
\email{smandujanov2000@alumno.ipn.mx}
\and
Victor Bautista-Ancona \at
Facultad de Matem\'aticas, UADY, M\'erida Yucat\'an, Mexico, 
\email{vbautista@correo.uady.mx}
}

\journalname{}

\begin{document}

\date{Received: \today / Accepted: date}

\maketitle
\begin{abstract}
Canteaut et al. demonstrated that bent Boolean functions are balanced over one of two complementary affine spaces. In this work, we provide an alternative proof of this property and leverage this result to establish the existence of $1$-plateaued functions whose restrictions to these affine spaces remain bent. By incorporating these results into the secondary Maiorana-McFarland construction, we obtain bent functions that are inherently balanced over a specific affine subspace (the underlying vector subspace). Furthermore, we analyse the balancedness of these new functions under linear perturbations. Finally, two algorithms are developed to simplify the research findings. It is worth noting that the sets of vectors with even and odd Hamming weights constitute a particular case of such complementary affine spaces.
\end{abstract}
\\
{\bf Kewords}: Boolean functions, Bent functions, Maiorana-McFarland, Balancedness, Affine spaces

\section{Introduction}
\label{Introduction}

In symmetric-key cryptography, the security of block ciphers and stream ciphers relies on the non-linear components used to achieve the property of confusion, as originally defined by Shannon \cite{Shannon}. In modern block ciphers, such as the Advanced Encryption Standard (AES), this non-linearity is typically provided by Substitution-boxes (S-boxes), which are constructed from Boolean functions with specific combinatorial properties. Among these, bent functions, introduced by Rothaus in 1976 \cite{RothausBent}, represent a class of functions that achieve the maximum distance from the set of all affine functions. This property, known as maximum non-linearity, makes them ideal for resisting linear cryptanalysis \cite{MatsuiLin}, as they minimise any linear correlation between input and output bits.
Bent functions intersect multiple disciplines; they are fundamental to the construction of Kerdock codes in coding theory and are linked to Golay complementary sequences in communication systems \cite{CarletBook}. However, despite their high non-linearity, bent functions possess a significant drawback for direct cryptographic implementation: they are unbalanced

A Boolean function is balanced if it takes the values $0$ and $1$ with equal probability. In contrast, a bent function in $n$ variables always has a Hamming weight of $2^{n-1} \pm 2^{n/2-1}$ \cite{TokarevaBent}. This lack of balance introduces a statistical bias that can be exploited by attackers to distinguish a ciphertext from a random sequence or to perform statistical attacks.

Interest in bent functions, despite their lack of balance, is due to their utility in the construction of other cryptographically significant functions. For instance, highly non-linear balanced functions are often derived by modifying bent functions locally or by using them as seeds \cite{CarletBook}. Furthermore, bent functions are closely related to the class of plateaued functions, which can achieve balancedness while maintaining a non-linearity near the optimal bound for balanced Boolean functions \cite{Cusick2017}.

Finding new bent functions remains a mathematical and computational challenge due to the size of the search space, which grows at a rate of $2^{2^n}$. Furthermore, even for dimensions as small as $n=8$, the immense size of the Boolean function space ($2^{256}$) makes exhaustive search \cite{Preneel1994} computationally infeasible, as bent functions represent a vanishingly small fraction of the total space. Consequently, research has focused on two main directions: algebraic constructions \cite{DillonTesis,Mesnager2016} and heuristic methods \cite{Clark2004,Maitra2002}. Algebraic approaches include primary constructions, such as the Maiorana-McFarland and Partial Spreads (PS) families \cite{DillonTesis}, and secondary constructions that extend bent functions from lower to higher dimensions. On the other hand, metaheuristic techniques, including Genetic Algorithms (GA), Simulated Annealing (SA), and Tabu Search, have been successfully employed to explore the search space and discover functions with specific properties that are difficult to obtain through purely algebraic means \cite{Maitra2002}.

The structural analysis of Boolean functions often requires examining their behaviour on restricted domains, particularly on affine spaces and hyperplanes \cite{CarletArxiv2025}. In this theory, affine spaces serve as fundamental sub-structures for function decomposition. As established in the work of Canteaut et al. \cite{Cosets}, the restriction of a bent function to certain affine hyperplanes reveals hidden symmetries and balancedness properties. This structural perspective is intimately linked to the concept of normality; a Boolean function is said to be normal if it remains constant on an affine subspace of dimension $n/2$ \cite{NormalBoolean}. The interplay between bentness and normality highlights the importance of affine spaces in understanding how cryptographic properties are distributed across a function's domain.

Furthermore, affine spaces are essential in the characterisation of plateaued functions \cite{ZhengZhang,Li2025TIT}. These functions are particularly pivotal to \textit{Construction K}, where bent functions of higher dimensions are generated from plateaued components whose Walsh spectrum support is an affine subspace \cite{ConstructionK}. Similarly, in \cite{GeneralizedPary}, affine spaces are employed to demonstrate that the function $f: \mathbb{F}_p^n \rightarrow \mathbb{F}_p^n$, where $p$ is a prime number, is partially bent if $g$ is a bent function. In this context, the relation $f(\bar{x}, \bar{y}) = g(\bar{x})$ is applied in a similar manner to that used for deriving plateaued functions in \textit{Construction K}.

Building on previous work \cite{KuCauich2025}, which utilised a function $g_e$ for general $r$ to generate bent functions restricted to even and odd Hamming weight spaces, this paper extends those results to general affine spaces (hyperplanes in this case). In the current study, we focus on the case $r=1$ by defining $g_e(x, x_s) = g(x)$, where $g_e: \mathbb{F}_2^s \rightarrow \mathbb{F}_2$ with $x \in \mathbb{F}_2^{s-1}$ and $x_s \in \mathbb{F}_2$. We demonstrate that a bent function $g$ generates new bent functions through the restrictions of $g_e$ onto the affine hyperplanes $H_{\bar{\alpha}}$ and $H_{\bar{\alpha}}^c$ (denoted as $g_{e_0}$ and $g_{e_1}$, respectively), where the choice of $\bar{\alpha} = \bar{1}$ recovers the specific case of Hamming weight parity. It is important to note that the existence of bent functions on these hyperplanes is also established in \cite{Cosets}, given the prior existence of a Boolean function satisfying condition (i) or (ii) of Theorem V.2 in the aforementioned work. In our research, we observe that our function $g_e$ satisfies both conditions (i) and (ii); in particular, it is a $1$-plateaued function. However, we additionally provide an alternative proof based on the Walsh-Hadamard transform. Another significant result is the balancedness of a bent function when restricting its domain to one of the previously mentioned hyperplanes \cite[Theorem V.3]{Cosets}. We provide a different and more direct proof based on the maximum non-linearity of the bent function and its non-linearity with respect to the sum of a linear function.

The previously discussed framework is applied to a specific instance of the Maiorana-McFarland secondary construction, as presented in \cite{CarletBook,CarletMaioranaExtendido}. Given a function $\phi: \mathbb{F}_2^s \rightarrow \mathbb{F}_2$ such that the preimage $\phi^{-1}(a)$ is an affine space of dimension $s-1$, and a function $g_e: \mathbb{F}_2^s \rightarrow \mathbb{F}_2$ where the restriction ${g_e}_{|\phi^{-1}(a)}$ is bent, a new bent function $f: \mathbb{F}_2^{s+1} \rightarrow \mathbb{F}_2$ is defined by:
\begin{equation}
(x_0, \bar{x}) \mapsto x_0 \phi(\bar{x}) \oplus g_e(\bar{x}),
\end{equation}
where $x_0 \in \mathbb{F}_2$ and $\bar{x} \in \mathbb{F}_2^s$.

The Maiorana-McFarland class is widely utilised to construct Boolean functions with diverse cryptographic properties, including normal functions \cite{TokarevaBent,GeneralizedPary}, resilient functions, and those with correlation immunity \cite{CarletBook,CamionCrypto}. Within this class, bent function constructions are categorised into primary and secondary types \cite{CarletBook}. Unlike secondary constructions, primary ones do not require the prior existence of bent functions. The specific secondary construction employed in this work belongs to the ${MM}_r$ class, as discussed in \cite{CarletBook}. In this research, we establish that bent functions obtained via this secondary construction are precisely balanced when restricted to the affine hyperplane $H_{\bar{\alpha}}$. Furthermore, to provide a comprehensive analysis of the restricted balancedness of bent functions, we examine their behaviour under the addition of linear functions, specifically through a theorem and a corollary that lead to a specific algorithm. Additionally, as a primary result of this research, we develop an algorithm based on the Maiorana-McFarland construction that generates bent functions of any dimension.

To clarify the scope and novelty of this book chapter, we explicitly distinguish between background material, alternate proofs of known results, and original additions. The core foundational material concerning general bent functions is recalled in Section \ref{Background} from established literature \cite{TokarevaBent,MacWilliams}. Similarly, Theorems \ref{Nonli},\ref{Parse}, and \ref{Afinbent} adapt well-known combinatorial properties and character sums to affine domains; although these adaptations are analogous to standard results, we explicitly include their proofs for mathematical completeness. In Section \ref{Balancednessbent}, Theorem \ref{balancedness} establishes the balancedness of bent functions over complementary affine spaces, for which we provide a significantly different proof compared to the one given by Canteaut et al. \cite{Cosets}. On page 10, we introduce an explicit matrix representation for the elements of $\mathbb{F}_2^s$; while a related representation for the general case of a positive integer $r$ was presented in \cite{KuCauich2025}, our specific case $r=1$ removes the restriction to even and odd Hamming weight sets, operating instead over general affine subspaces. For Theorem \ref{BentBalancness}, we present a direct proof utilising the Walsh-Hadamard transform, although we note that this theorem could also be established by observing that the extended function $g_e$ is 1-plateaued and thus applying Theorem V.2 from \cite{Cosets}. Furthermore, within our proof of Theorem \ref{BentBalancness}, while Case 2 can be resolved via Theorem \ref{balancedness}, we additionally provide an alternative proof for this specific case. From this framework, we derive two original and specific consequences: Corollary \ref{BentWalsh1} and Corollary \ref{balanceextende}. Theorem \ref{balancedextended} introduces an original secondary construction using the Maiorana-McFarland framework which explicitly diverges from the design in \cite{KuCauich2025} because the underlying affine subspaces are no longer constrained to vectors of even or odd Hamming weights. Moreover, Theorem \ref{UsingMaiorana} provides an original result that extends the balancedness analysis to linear perturbations by adding affine functions to the newly generated bent functions. Finally, we contribute two original algorithms where the first constructs new families of bent functions, while the second determines and ensures restricted balancedness under specific conditions.

We proceed as follows. In Section \ref{Background}, we review the necessary preliminaries and fundamental definitions of Boolean functions. In Section \ref{Balancednessbent}, we provide an alternative proof of the property that bent functions are balanced when restricted to specific affine spaces; this analysis illustrates the distribution of pre-images for both $0$ and $1$ across the entire domain. Section \ref{NewFamily} introduces two complementary affine subspaces of $\mathbb{F}_2^s$ and demonstrates the existence of specific bent functions defined over them, derived from an initial bent function on $\mathbb{F}_2^{s-1}$. This section also analyses the addition of linear functions to the resulting bent functions. In Section \ref{Secondary Constrution}, the previously defined functions and affine spaces are integrated into the secondary Maiorana-McFarland construction to generate bent functions that are balanced exclusively on one of the affine spaces. Similar to the preceding section, we evaluate the restricted balancedness of these new functions under linear perturbations. Finally, we provide two algorithms: one to construct new bent functions of any dimension by iteratively applying Theorem \ref{balancedextended}, and another to construct bent functions under linear sums while predicting the specific affine space where balancedness occurs.

\section{Background}
\label{Background}

This section recalls the fundamental definitions and results concerning Boolean functions, with a particular focus on bent functions. For a more comprehensive treatment of these topics, the reader is referred to \cite{TokarevaBent}, \cite{MacWilliams}, and \cite{BFunctionsRepresentationCarlet}.

\begin{definition}
A function $f: \mathbb{F}_2^n \rightarrow \mathbb{F}_2$ is called a \textbf{Boolean function}. We denote the set of all Boolean functions with domain $\mathbb{F}_2^n$ by $\mathcal{B}_n$.
\end{definition} 

Any Boolean function $f \in \mathcal{B}_n$ can be uniquely represented by its \textbf{algebraic normal form} (ANF):
\[ f(\bar{x}) = \bigoplus_{\bar{u} \in \mathbb{F}_2^n} a_{\bar{u}} \bar{x}^{\bar{u}}, \]
where $a_{\bar{u}} \in \mathbb{F}_2$ and the term $\bar{x}^{\bar{u}}$ is defined as the product $x_1^{u_1} \cdots x_n^{u_n}$ for $\bar{x} = (x_1, \dots, x_n)$ and $\bar{u} = (u_1, \dots, u_n)$.

\begin{example}
The Boolean function $f \in \mathcal{B}_3$ defined by $f(x_1, x_2, x_3) = 1 \oplus x_1 x_2 \oplus x_1 x_2 x_3$ is presented in its ANF.
\end{example}

The coefficients $a_{\bar{u}}$ of the ANF can be determined from the values of the function as follows:

\begin{theorem}
Let $f \in \mathcal{B}_n$. The coefficients $a_{\bar{u}}$ of its algebraic normal form are given by:
\[ a_{\bar{u}} = \bigoplus_{\bar{x} \leq \bar{u}} f(\bar{x}), \]
where $\bar{x} \leq \bar{u}$ denotes the partial order such that $x_i \leq u_i$ for all $i \in \{1, \dots, n\}$.
\end{theorem}

\begin{definition}
The set of all \textbf{affine} Boolean functions from $\mathbb{F}_2^n$ to $\mathbb{F}_2$, denoted by $\mathcal{A}_n$, is defined as:
\[ \mathcal{A}_n := \{ l_{\bar{a}, a_0} \mid \bar{a} \in \mathbb{F}_2^n, a_0 \in \mathbb{F}_2 \}, \]
where $l_{\bar{a}, a_0}(\bar{x}) = \bar{a} \cdot \bar{x} \oplus a_0$ and $\cdot$ denotes the standard inner product over $\mathbb{F}_2^n$.
\end{definition}

\begin{example}
Consider the case $n=2$ with the domain $\mathbb{F}_2^2$. The set of linear functions has cardinality $2^2 = 4$ by setting $a_0 = 0$:
\[ l_{(0,0)}(\bar{x}) = 0, \quad l_{(1,0)}(\bar{x}) = x_1, \quad l_{(0,1)}(\bar{x}) = x_2, \quad l_{(1,1)}(\bar{x}) = x_1 \oplus x_2. \]
By choosing $a_0 = 1$, we obtain their corresponding affine shifts:
\[ l_{(0,0), 1}(\bar{x}) = 1, \quad l_{(1,0), 1}(\bar{x}) = x_1 \oplus 1, \quad l_{(0,1), 1}(\bar{x}) = x_2 \oplus 1, \quad l_{(1,1), 1}(\bar{x}) = x_1 \oplus x_2 \oplus 1. \]
Thus, the total set of affine functions $\mathcal{A}_2$ contains exactly $2^{2+1} = 8$ elements.
\end{example}

In the specific case where $a_0 = 0$, the function $l_{\bar{a}, 0}$ is referred to as a \textbf{linear function} and is denoted simply by $l_{\bar{a}}$. It is worth noting that $|\mathcal{A}_n| = 2^{n+1}$, whereas the number of linear functions is $2^n$. To put these quantities into perspective, the total number of general Boolean functions from $\mathbb{F}_2^n$ to $\mathbb{F}_2$ is $2^{2^n}$. Therefore, as $n$ grows, the set of affine functions is significantly smaller.

\begin{definition}
The \textbf{Hamming weight} of a vector $\bar{x} \in \mathbb{F}_2^n$, denoted by $w_H(\bar{x})$, is the number of its non-zero coordinates. Correspondingly, the \textbf{Hamming distance} between two Boolean functions $f, g \in \mathcal{B}_n$, denoted by $d_H(f, g)$, is defined as the number of positions where their values differ:
\[ d_H(f, g) := |\{ \bar{x} \in \mathbb{F}_2^n \mid f(\bar{x}) \neq g(\bar{x}) \}| = w_H(f \oplus g). \]
\end{definition}

\begin{definition}
The \textbf{non-linearity} of a Boolean function $f \in \mathcal{B}_n$, denoted by $\mathcal{N}(f)$, is defined as the minimum Hamming distance between $f$ and the set of all affine functions:
\[ \mathcal{N}(f) := \min_{g \in \mathcal{A}_n} d_H(f, g). \]
Boolean functions that achieve the maximum possible non-linearity for a given $n$ are termed \textbf{bent functions}.
\end{definition}

To further characterise non-linearity, we utilise the following transform:

\begin{definition}
The \textbf{Walsh-Hadamard Transform} of a Boolean function $f \in \mathcal{B}_n$ at $\bar{a} \in \mathbb{F}_2^n$ is defined as:
\[ \widehat{\mathcal{W}}_f(\bar{a}) := \sum_{\bar{x} \in \mathbb{F}_2^n} (-1)^{f(\bar{x}) \oplus \bar{a} \cdot \bar{x}}. \]
\end{definition}

The non-linearity of a function can be explicitly expressed in terms of its Walsh-Hadamard coefficients.

\begin{theorem}
The non-linearity of a Boolean function $f \in \mathcal{B}_n$ is given by:
\[ \mathcal{N}(f) = 2^{n-1} - \frac{1}{2} \max_{\bar{a} \in \mathbb{F}_2^n} |\widehat{\mathcal{W}}_f(\bar{a})|. \]
\end{theorem}

\begin{theorem}
A Boolean function $f \in \mathcal{B}_n$ is bent if and only if its Walsh-Hadamard coefficients satisfy $\widehat{\mathcal{W}}_f(\bar{a}) = \pm 2^{n/2}$ for every $\bar{a} \in \mathbb{F}_2^n$.
\end{theorem}

Consequently, bent functions exist only for even $n$ and possess a non-linearity of $2^{n-1} - 2^{n/2-1}$.

\section{The balancedness of bent functions in two affine subspaces of their domain}
\label{Balancednessbent}

In this section, we investigate the distribution of the preimages of a bent function over two affine subspaces of $\mathbb{F}_2^n$: the hyperplane $H_{\bar{\alpha}} := \{ \bar{x} \in \mathbb{F}_2^n \mid l_{\bar{\alpha}}(\bar{x}) = 0 \}$ and its complement $H_{\bar{\alpha}}^c := \{ \bar{x} \in \mathbb{F}_2^n \mid l_{\bar{\alpha}}(\bar{x}) = 1 \}$, where $\bar{\alpha} \neq \bar{0}$. Note that $H_{\bar{\alpha}}$ is an $(n-1)$-dimensional vector space, a property derived from the balancedness and linearity of the function $l_{\bar{\alpha}}$. Specifically, if $\bar{\alpha} = \bar{1} = (1, \dots, 1)$, these affine spaces correspond to the set of vectors with even Hamming weight ($H_{\bar{1}}$) and the set of vectors with odd Hamming weight ($H_{\bar{1}}^c$).

Theorem \ref{balancedness} below coincides with Theorem V.3 established by Canteaut et al. \cite{Cosets}. However, the underlying methodology differs significantly. While the proof in \cite{Cosets} relies on an analytical approach supported by their Lemma B.1 (which employs Jacobi's theta functions to constrain the possible integer solutions of the sum of squares $X^2 + Y^2 = 2^{n+1}$) the proof presented here adopts a combinatorial-structural perspective. 

Our approach directly exploits the definition of Hamming weight and the geometric partition of the space $\mathbb{F}_2^n$ into affine hyperplanes. Furthermore, a fundamental insight in our demonstration is the consideration of the bent function $f \oplus l_{\bar{\alpha}}$ derived from the original bent function $f$.

\begin{theorem}\label{balancedness} 
Let $H_{\bar{\alpha}} := \{ \bar{x} \in \mathbb{F}_2^n \mid l_{\bar{\alpha}}(\bar{x}) = 0 \}$ and $H_{\bar{\alpha}}^c := \{ \bar{x} \in \mathbb{F}_2^n \mid l_{\bar{\alpha}}(\bar{x}) = 1 \}$. Every bent function $f: \mathbb{F}_2^n \rightarrow \mathbb{F}_2$, where $n \geq 2$ and $n$ is even, satisfies that either $f_{|H_{\bar{\alpha}}}$ or $f_{|H_{\bar{\alpha}}^c}$ is balanced.
\end{theorem}

\begin{proof}

Let $f \in \mathcal{B}_n$ be a bent function. From its non-linearity characterization and Walsh-Hadamard transform, it follows that for any $\bar{a} \in \mathbb{F}_2^n$:
\[ w_H(f \oplus l_{\bar{a}}) = 2^{n-1} \pm 2^{\frac{n-2}{2}}. \]
Since $l_{\bar{\alpha}}(H_{\bar{\alpha}}) = \{ 0 \}$ and $l_{\bar{\alpha}}(H_{\bar{\alpha}}^c) = \{ 1 \}$, the weight distributions across the partition $\mathbb{F}_2^n = H_{\bar{\alpha}} \cup H_{\bar{\alpha}}^c$ satisfy:
\begin{align*}
w_H((f \oplus l_{\bar{\alpha}})_{|H_{\bar{\alpha}}}) &= w_H(f_{|H_{\bar{\alpha}}}), \\
w_H((f \oplus l_{\bar{\alpha}})_{|H_{\bar{\alpha}}^c}) &= |H_{\bar{\alpha}}^c| - w_H(f_{|H_{\bar{\alpha}}^c}) = 2^{n-1} - w_H(f_{|H_{\bar{\alpha}}^c}).
\end{align*}
Setting $w_H(f_{|H_{\bar{\alpha}}}) = c$ with $0 \leq c \leq 2^{n-1}$, the total weight $w_H(f) = w_H(f_{|H_{\bar{\alpha}}}) + w_H(f_{|H_{\bar{\alpha}}^c}) = 2^{n-1} \pm 2^{\frac{n-2}{2}}$ implies that $w_H(f_{|H_{\bar{\alpha}}^c})$ must satisfy one of the following two conditions:
\[ w_H(f_{|H_{\bar{\alpha}}^c}) = \left(2^{n-1} \pm 2^{\frac{n-2}{2}}\right) - c. \]

\textbf{Case 1.} Let $w_H(f_{|H_{\bar{\alpha}}^c}) = (2^{n-1} - 2^{\frac{n-2}{2}}) - c$. Then, we have $w_H((f \oplus l_{\bar{\alpha}})_{|H_{\bar{\alpha}}}) = c$ and $w_H((f \oplus l_{\bar{\alpha}})_{|H_{\bar{\alpha}}^c}) = 2^{n-1} - (2^{n-1} - 2^{\frac{n-2}{2}} - c) = 2^{\frac{n-2}{2}} + c$. By considering the sum of all weights, we obtain $w_H(f \oplus l_{\bar{\alpha}}) = c + 2^{\frac{n-2}{2}} + c = 2^{\frac{n-2}{2}} + 2c$.

\begin{itemize}
\item \mbox{\bf{Subcase 1a.}} If $w_H(f \oplus l_{\bar{\alpha}}) = 2^{n-1} - 2^{\frac{n-2}{2}}$, we find that $c = 2^{n-2} - 2^{\frac{n-2}{2}}$. This results in $w_H(f_{|H_{\bar{\alpha}}}) = 2^{n-2} - 2^{\frac{n-2}{2}}$ and $w_H(f_{|H_{\bar{\alpha}}^c}) = 2^{n-2}$. Therefore, $f_{|H_{\bar{\alpha}}^c}$ is balanced.
    
    \item \mbox{{\bf Subcase 1b.}} If $w_H(f \oplus l_{\bar{\alpha}}) = 2^{n-1} + 2^{\frac{n-2}{2}}$, we find that $c = 2^{n-2}$. This results in $w_H(f_{|H_{\bar{\alpha}}}) = 2^{n-2}$ and $w_H(f_{|H_{\bar{\alpha}}^c}) = 2^{n-2} - 2^{\frac{n-2}{2}}$. Therefore, $f_{|H_{\bar{\alpha}}}$ is balanced.
\end{itemize}

\textbf{Case 2.} Let $w_H(f_{|H_{\bar{\alpha}}^c}) = (2^{n-1} + 2^{\frac{n-2}{2}}) - c$. Then, $w_H((f \oplus l_{\bar{\alpha}})_{|H_{\bar{\alpha}}}) = c$ and $w_H((f \oplus l_{\bar{\alpha}})_{|H_{\bar{\alpha}}^c}) = 2^{n-1} - (2^{n-1} + 2^{\frac{n-2}{2}} - c) = -2^{\frac{n-2}{2}} + c$. Considering the sum of all weights, we have $w_H(f \oplus l_{\bar{\alpha}}) = c - 2^{\frac{n-2}{2}} + c = -2^{\frac{n-2}{2}} + 2c$.

\begin{itemize}
    \item  \mbox{\bf Subcase 2a.} If $w_H(f \oplus l_{\bar{\alpha}}) = 2^{n-1} - 2^{\frac{n-2}{2}}$, we find that $c = 2^{n-2}$. This results in $w_H(f_{|H_{\bar{\alpha}}}) = 2^{n-2}$ and $w_H(f_{|H_{\bar{\alpha}}^c}) = 2^{n-2} + 2^{\frac{n-2}{2}}$. Therefore, $f_{|H_{\bar{\alpha}}}$ is balanced.
    
    \item \mbox{\bf Subcase 2b.} If $w_H(f \oplus l_{\bar{\alpha}}) = 2^{n-1} + 2^{\frac{n-2}{2}}$, we find that $c = 2^{n-2} + 2^{\frac{n-2}{2}}$. This results in $w_H(f_{|H_{\bar{\alpha}}}) = 2^{n-2} + 2^{\frac{n-2}{2}}$ and $w_H(f_{|H_{\bar{\alpha}}^c}) = 2^{n-2}$. Therefore, $f_{|H_{\bar{\alpha}}^c}$ is balanced.
\end{itemize}

Since all possible cases lead to balancedness when restricting the domain of the bent function to either even or odd Hamming weight vectors, the proof is complete.
\end{proof}

It is important to note that any Boolean function satisfying the Hamming weight conditions described in the aforementioned subcases is necessarily a bent function. Furthermore, the proof provided not only establishes the balancedness of the preimages over one of the affine subspaces but also determines the exact number of preimages across the entire domain of the function, depending on the specific subcase considered. Finally, this approach highlights the direct structural relationship between the bent function $f$ and its affine shift $f \oplus l_{\bar{\alpha}}$.

\medskip
Based on the maximum non-linearity of a bent function and Theorem \ref{balancedness}, the cardinality distribution of the preimages of a bent function is as follows:

\begin{remark}\label{DistriWeight}
Let $f \in \mathcal{B}_n$ be a bent function. The cardinality of its preimages over the partition $\mathbb{F}_2^n = H_{\bar{\alpha}} \cup H_{\bar{\alpha}}^c$ is symmetrically determined by the sign of $\widehat{\mathcal{W}}_{f}(\bar{0})$ and the specific subspace onto which $f$ is balanced:

\small
\begin{center}
\begin{tabular}{|c|c|c|c|c|c|}
\hline
\textbf{Balanced Subspace} & $\widehat{\mathcal{W}}_{f}(\bar{0})$ & $\#(f)_{|H_{\bar{\alpha}}}^{-1}(0)$ & $\#(f)_{|H_{\bar{\alpha}}}^{-1}(1)$ & $\#(f)_{|H_{\bar{\alpha}}^c}^{-1}(0)$ & $\#(f)_{|H_{\bar{\alpha}}^c}^{-1}(1)$ \\ \hline
\hline
\rule{0pt}{3ex} & $2^{n/2}$  & $2^{n-2}$ & $2^{n-2}$ & $2^{n-2} + 2^{\frac{n-2}{2}}$ & $2^{n-2} - 2^{\frac{n-2}{2}}$ \\ [1ex]
\cline{2-6}
\raisebox{1.5ex}[0pt]{$f_{|H_{\bar{\alpha}}}$ is balanced} & $-2^{n/2}$ & $2^{n-2}$ & $2^{n-2}$ & $2^{n-2} - 2^{\frac{n-2}{2}}$ & $2^{n-2} + 2^{\frac{n-2}{2}}$ \\ \hline
\hline
\rule{0pt}{3ex} & $2^{n/2}$  & $2^{n-2} + 2^{\frac{n-2}{2}}$ & $2^{n-2} - 2^{\frac{n-2}{2}}$ & $2^{n-2}$ & $2^{n-2}$ \\ [1ex]
\cline{2-6}
\raisebox{1.5ex}[0pt]{$f_{|H_{\bar{\alpha}}^c}$ is balanced} & $-2^{n/2}$ & $2^{n-2} - 2^{\frac{n-2}{2}}$ & $2^{n-2} + 2^{\frac{n-2}{2}}$ & $2^{n-2}$ & $2^{n-2}$ \\ \hline
\end{tabular}
\end{center}
\normalsize

It can be observed that when $f_{|H_{\bar{\alpha}}}$ is balanced and $\widehat{\mathcal{W}}_{f}(\bar{0}) = 2^{n/2}$, the number of preimages of $f$ mapped to one over $H_{\bar{\alpha}}^c$ decreases to $\#(f)_{|H_{\bar{\alpha}}^c}^{-1}(1) = 2^{n-2} - 2^{\frac{n-2}{2}}$. Under these conditions, the complement relation $\#(f \oplus l_{\bar{\alpha}})_{|H_{\bar{\alpha}}^c}^{-1}(1) = 2^{n-1} - \#(f)_{|H_{\bar{\alpha}}^c}^{-1}(1)$ established in the previous proof implies that the number of preimages mapped to one for $f \oplus l_{\bar{\alpha}}$ over $H_{\bar{\alpha}}^c$ increases to $2^{n-2} + 2^{\frac{n-2}{2}}$, as verified in Subcase 1b.
\end{remark}

\section{Construction of bent functions over two affine spaces}
\label{NewFamily}

We extend the definition of bent functions with domain ${\mathbb F}_2^n$ to include bent functions defined on an affine subspace, as utilised in \cite{Cosets} and \cite{CarletMaioranaExtendido}. Subsequently, we represent the vector space ${\mathbb F}_2^s$ in terms of two complementary affine spaces. We then identify bent functions over these spaces and analyse the addition of linear functions to the resulting constructions.

\medskip
The following three theorems extend the classical properties of Boolean functions to the context of affine domains. Although their proofs follow standard combinatorial and algebraic arguments based on character sums analogous to those over the full space ${\mathbb F}_2^n$ (see, e.g., \cite{CarletBook,MacWilliams}), we briefly outline these proofs in order to establish the explicit dimensional dependencies over the restricted subspace $\mathcal{C}$ (cf. \cite{carlet1994two,CharpinRestriction}).

\begin{theorem}\label{Nonli} Let a function $f:{\mathcal C}\subseteq {\mathbb F}_2^n \rightarrow {\mathbb F}_2$. Then $${\mathcal N}_{\mathcal C}(f) = 2^{m-1}- \frac{1}{2}\max_{\bar{a}\in {\mathbb F}_2^n} |\widehat{{\mathcal W}}_{f}(\bar{a})|,$$ 

where $\widehat{{\mathcal W}}_f(\bar{a}):= \sum \limits_{\bar{x} \in {\mathcal C}} (-1)^{f(\bar{x})\oplus \bar{a}\cdot \bar{x}}.$
\end{theorem}

\begin{proof} We can see that, for all $\bar{a}\in{\mathbb F}_2^n$,  $$\widehat{{\mathcal W}}_f(\bar{a})=2^{m}-2d_H(f,\bar{a}\cdot \bar{x}) \mbox{ and } -\widehat{{\mathcal W}}_f(\bar{a})=2^{m}-2d_H(f,\bar{a}\cdot \bar{x}\oplus 1).$$ Resolving $d_H$ on the left side, we obtain the result in both cases.

\end{proof}

\begin{theorem}\label{Parse}(Parseval's equation) Let $f:{\mathcal C}\subseteq {\mathbb F}_2^n \rightarrow {\mathbb F}_2$. Then, $$\sum_{\bar{a}\in {\mathbb F}_2^n} {\widehat{{\mathcal W}}}_f^2(\bar{a}) = 2^{m+n}.$$
\end{theorem}

\begin{proof}Resolving,
$$\sum_{\bar{a}\in {\mathbb F}_2^n} \widehat{{\mathcal W}}_f(\bar{a}){\widehat{\mathcal W}}_f(\bar{a})=\sum_{\bar{x},\bar{y}\in {{\mathcal C}}}(-1)^{f(\bar{x})\oplus f(\bar{y})} \sum_{\bar{a}\in {\mathbb F}_2^n}(-1)^{\bar{a}\cdot (\bar{x}\oplus \bar{y})}= 2^{m+n}.$$ 

The balancedness of the linear functions, together with the distinct elements $\bar{x}=\bar{y}$, yields the desired result.

\end{proof}

\begin{definition}\label{def:bent_affine}
The nonlinearity of a function $f:{\mathcal C}\rightarrow {\mathbb F}_2$ over an affine space ${\mathcal C} \subseteq {\mathbb F}_2^n$ ($\dim {\mathcal C}=m \leq n$) is defined as its minimum Hamming distance to the set of restricted affine functions:
\begin{equation}
    {\mathcal N}_{\mathcal C}(f):= \min \limits_{l\in {\mathcal A}_n} d_H(f, l_{|{\mathcal C}}).
\end{equation}
The function $f$ is bent over ${\mathcal C}$ if its nonlinearity achieves the maximum possible value.
\end{definition}

\begin{theorem}\label{Afinbent}
The function $f: {\mathcal C} \subseteq \mathbb{F}_2^n \rightarrow \mathbb{F}_2$, with $\dim {\mathcal C} = m$, is bent if and only if $|\widehat{{\mathcal W}}_f(\bar{a})| = 2^{m/2}$ for all $\bar{a} \in \mathbb{F}_2^n$.
\end{theorem}

\begin{proof}
According to Theorem \ref{Nonli}, a function is bent if the maximum value of its Walsh-Hadamard transform, $|\widehat{{\mathcal W}}_f|$, is as small as possible. We shall demonstrate that this minimum possible value is $2^{m/2}$.

Suppose that $\max_{\bar{a} \in \mathbb{F}_2^n} |\widehat{{\mathcal W}}_f(\bar{a})| < 2^{m/2}$. By Parseval's identity, which states that $\sum_{\bar{a} \in \mathbb{F}_2^n} (\widehat{{\mathcal W}}_f(\bar{a}))^2 = 2^{m+n}$, it would follow that there must exist some $\bar{a}' \in \mathbb{F}_2^n$ such that $|\widehat{{\mathcal W}}_f(\bar{a}')| > 2^{m/2}$ to satisfy the sum, which constitutes a contradiction.

Consequently, the minimum value that $\max_{\bar{a} \in \mathbb{F}_2^n} |\widehat{{\mathcal W}}_f(\bar{a})|$ can attain is $2^{m/2}$. Furthermore, Parseval's equality implies that if the maximum value is $2^{m/2}$, then every transform coefficient must satisfy $\widehat{{\mathcal W}}_f(\bar{a}) = \pm 2^{m/2}$ for all $\bar{a} \in \mathbb{F}_2^n$.
\end{proof}

Hence, bent functions over affine subspaces only exist when the dimension $m$ of the affine subspace is even, and they possess a non-linearity equal to $\mathcal{N}_f = 2^{m-1} - 2^{\frac{m}{2}-1}.$

\begin{example}
The existence of bent functions over affine spaces is guaranteed for various values of $n$. For instance, let $n=5$. The Boolean function 
\[ f(x_1, x_2, x_3, x_4, x_5) = x_2x_3 \oplus x_1x_4 \oplus x_1x_3 \oplus x_1x_2 \]
is an example of a bent function defined over the subspace of even Hamming weight vectors in $\mathbb{F}_2^5$. In this case, $|\widehat{{\mathcal W}}_f(\bar{a})| = 2^{4/2} = 4$ for all $\bar{a} \in \mathbb{F}_2^5$.
\end{example}

\medskip
The following theorem corresponds to a class of bent functions known as the \textbf{secondary Maiorana-McFarland construction}.

\begin{theorem}{\label{Maiorana}}\cite{CarletMaioranaExtendido}
Let the function $\phi(\bar{y}): \mathbb{F}_2^s \rightarrow \mathbb{F}_2^r$ be defined such that for all $\bar{a} \in \mathbb{F}_2^r$, the preimage $\phi^{-1}(\bar{a})$ forms an affine space of dimension $s - r$. Additionally, let $g(\bar{y}): \mathbb{F}_2^s \rightarrow \mathbb{F}_2$ be a function for which the restriction $g_{|\phi^{-1}(\bar{a})}$ is a bent function. We then define the function $f : \mathbb{F}_2^{r+s} \rightarrow \mathbb{F}_2$ by $f(\bar{x},\bar{y}) = \bar{x} \cdot \phi(\bar{y}) \oplus g(\bar{y})$, where $\bar{x} \in \mathbb{F}_2^r$. Under these conditions, the function $f$ is bent.
\end{theorem}

In the preceding theorem, we observe that bent functions defined over affine spaces of dimension $s-r$ can be used to construct bent functions over the vector space $\mathbb{F}_2^{r+s}$, which has dimension $r+s$.

The sets, vectors, and functions defined below follow this notation throughout the remainder of this work. To apply Theorem \ref{Maiorana} to the case $r=1$, we consider the following affine spaces:

$$H_{\bar{\alpha}} := \{\bar{x} \in \mathbb{F}_2^s \mid l_{\bar{\alpha}}(\bar{x}) = 0\} \quad \text{and} \quad H_{\bar{\alpha}}^c := \{\bar{x} \in \mathbb{F}_2^s \mid l_{\bar{\alpha}}(\bar{x}) = 1\}.$$ Similarly, we denote the sets $H_{\alpha}$ and $H_{\alpha}^c$ as:
\[ H_{\alpha} := \{x \in \mathbb{F}_2^{s-1} \mid l_{\alpha}(x) = 0\} \quad \text{and} \quad H_{\alpha}^c := \{x \in \mathbb{F}_2^{s-1} \mid l_{\alpha}(x) = 1\}, \]
where $\bar{x} = (x, x_s)$ and $\bar{\alpha} = (\alpha, \alpha_s)$, with $x, \alpha \in \mathbb{F}_2^{s-1}$ and $x_s, \alpha_s \in \mathbb{F}_2$.

Note that $H_{\bar{\alpha}}$ is a linear subspace with dimension $s-1$, and $H_{\bar{\alpha}}^c$ is an affine space such that for any $\bar{b} \notin H_{\bar{\alpha}}$, we have $\bar{b} \oplus H_{\bar{\alpha}} = H_{\bar{\alpha}}^c$. Furthermore, $H_{\alpha}$ is a linear subspace with dimension $s-2$, and $H_{\alpha}^c$ is an affine space such that for any $b \notin H_{\alpha}$, we have $b \oplus H_{\alpha} = H_{\alpha}^c$. 

By an abuse of notation, we represent the elements of $\mathbb{F}_2^s$ as matrices, where each row corresponds to an element of $\mathbb{F}_2^s$. Let $\bar{\alpha}=(\alpha_1,\ldots, \alpha_{s-1}, \alpha_s)$. Without loss of generality, we assume that $\alpha_s \neq 0$. Consequently, the elements of $\mathbb{F}_2^s$ can be partitioned as follows:

$$\left[ \begin{array}{c} H_{\bar{\alpha}} \\ \hline \rule{0pt}{2.4ex} H_{\bar{\alpha}}^c \end{array} \right] = 
\left[ \begin{array}{cc} H_{\alpha} & \bar{0} \\ H_{\alpha}^c & \bar{1} \\ \hline \rule{0pt}{2.4ex} H_{\alpha} & \bar{1} \\ H_{\alpha}^c & \bar{0} \end{array} \right],$$
where $\bar{0}$ and $\bar{1}$ are column vectors of dimension $2^{s-2} \times 1$ consisting entirely of zeros and ones, respectively. 

Observe that the sets $H_{\bar{\alpha}}$ and $H_{\bar{\alpha}}^c$ form a partition of ${\mathbb F}_2^s$, since $l_{\bar{\alpha}}$ is balanced due to its property as a linear function. Furthermore, in the matrix representation on the right-hand side, the columns $\bar{0}$ and $\bar{1}$ are required in the given order to preserve the corresponding images of $l_{\bar{\alpha}}$.

Let $g: \mathbb{F}_2^{s-1} \rightarrow \mathbb{F}_2$ be a bent function. From now on, $g_e: \mathbb{F}_2^s \rightarrow \mathbb{F}_2$ and $\phi: \mathbb{F}_2^s \rightarrow \mathbb{F}_2$ are functions (with the aim of applying Theorem \ref{Maiorana} to the specific case $r=1$) defined as follows:

\begin{itemize}
    \item $\phi^{-1}(0) = H_{\bar{\alpha}}$ and $\phi^{-1}(1) = H_{\bar{\alpha}}^c$. That is, $\phi = l_{\bar{\alpha}}$.
    \item The restrictions of $g_e$ are given by ${g_e}_{|H_{\bar{\alpha}}} := g_{e_0}$ and ${g_e}_{|H_{\bar{\alpha}}^c} := g_{e_1}$, where:
    \begin{itemize}
        \item $g_{e_0}: H_{\bar{\alpha}} \rightarrow \mathbb{F}_2$ is defined as $g_{e_0}(x, x_s) := g(x)$.
        \item $g_{e_1}: H_{\bar{\alpha}}^c \rightarrow \mathbb{F}_2$ is defined as $g_{e_1}(x, x_s) := g(x)$.
    \end{itemize}
\end{itemize}

\begin{remark}\label{observaciones}
In $H_{\bar{\alpha}}$, if $x \in \mathbb{F}_2^{s-1}$ satisfies $l_{\alpha}(x) = 0$, then $x_s = 0$. If $l_{\alpha}(x) = 1$, then $x_s = 1$. Conversely, in $H_{\bar{\alpha}}^c$, if $x \in \mathbb{F}_2^{s-1}$ satisfies $l_{\alpha}(x) = 0$, then $x_s = 1$. If $l_{\alpha}(x) = 1$, then $x_s = 0$.
\end{remark}

Theorem \ref{BentBalancness} explores the behaviour of bent functions when restricted to affine spaces. An alternative way to prove the assertion of this theorem is by observing that our function $g_e$ satisfies the conditions for ${\mathcal H}_{\alpha}$ according to the definition in \cite{Cosets}, being three-valued and almost-optimal as per condition (ii) of Theorem V.2 in \cite{Cosets}. Thus, by applying the aforementioned Theorem V.2, the assertion of our next theorem is resolved. In their work, the authors establish the bentness of the restricted function as a consequence of utilising the properties of Boolean derivatives and their Lemma B.1. This lemma employs the theory of Jacobi theta series and generating functions to characterise the number of representations of an integer as a sum of two squares. 

In contrast, by analysing the geometric partition of the space $\mathbb{F}_2^s$, the algebraic properties of the Walsh-Hadamard transform, and Theorem \ref{balancedness}, we prove that the bent property is preserved under these restrictions through a direct verification of the Fourier spectrum. 

\begin{theorem}\label{BentBalancness} 
Let $g: \mathbb{F}_2^{s-1} \rightarrow \mathbb{F}_2$ be a bent function. Then, the restricted functions $g_{e_0}: H_{\bar{\alpha}} \rightarrow \mathbb{F}_2$ and $g_{e_1}: H_{\bar{\alpha}}^c \rightarrow \mathbb{F}_2$ are also bent functions.
\end{theorem}

\begin{proof}
Let $\bar{a}=(a, a_s) \in \mathbb{F}_2^s$ and $\bar{x}=(x, x_s) \in H_{\bar{\alpha}}$. We denote the linear form as $l_a(x) = x \cdot a$, where $a, x \in \mathbb{F}_2^{s-1}$. The Walsh-Hadamard transform of $g_{e_0}$ is given by:
\begin{equation*}
\widehat{{\mathcal W}}_{g_{e_0}}(\bar{a}) = \sum_{\bar{x} \in H_{\bar{\alpha}}} (-1)^{g_{e_0}(\bar{x}) \oplus \bar{x} \cdot \bar{a}} = \sum_{(x, x_s) \in H_{\bar{\alpha}}} (-1)^{g_{e_0}(x, x_s) \oplus x \cdot a \oplus x_s a_s}.
\end{equation*}

\textbf{Case 1:} If $a_s = 0$, then:
\begin{equation*}
\widehat{{\mathcal W}}_{g_{e_0}}(a, 0) = \sum_{x \in \mathbb{F}_2^{s-1}} (-1)^{g(x) \oplus x \cdot a} = \widehat{{\mathcal W}}_{g}(a).
\end{equation*}
Since $g$ is a bent function, $|\widehat{{\mathcal W}}_{g}(a)| = 2^{(s-1)/2}$, which implies that $g_{e_0}$ satisfies the bent condition for this case.

\textbf{Case 2:} If $a_s = 1$, we utilise the previously established matrix ordering of $\mathbb{F}_2^s$ relative to $H_{\bar{\alpha}}$ and $H_{\bar{\alpha}}^c$:
\begin{eqnarray*}
\widehat{\mathcal W}_{g_{e_0}}(a, 1) &=& \sum_{(x, 0) \in H_{\bar{\alpha}}} (-1)^{g_{e_0}(x, 0) \oplus x \cdot a} + \sum_{(x, 1) \in H_{\bar{\alpha}}} (-1)^{g_{e_0}(x, 1) \oplus x \cdot a \oplus 1} \\
&=& \sum_{x \in H_{\alpha}} (-1)^{g(x) \oplus x \cdot a} - \sum_{x \in H_{\alpha}^c} (-1)^{g(x) \oplus x \cdot a}.
\end{eqnarray*}
The last equality follows from Remark \ref{observaciones}. We now consider the following subcases:
\begin{itemize}
    \item \textbf{Subcase 2a:} If $(g \oplus l_a)_{|H_{\alpha}}$ is balanced, then $\widehat{\mathcal W}_{g_{e_0}}(a, 1) = -\widehat{\mathcal W}_{g}(a)$.
    \item \textbf{Subcase 2b:} If $(g \oplus l_a)_{|H_{\alpha}^c}$ is balanced, then $\widehat{\mathcal W}_{g_{e_0}}(a, 1) = \widehat{\mathcal W}_{g}(a)$.
\end{itemize}
In both subcases, since $|\widehat{\mathcal W}_{g}(a)| = 2^{(s-1)/2}$, it follows that $|\widehat{\mathcal W}_{g_{e_0}}(a, 1)| = 2^{(s-1)/2}$. Thus, $g_{e_0}$ is a bent function. By symmetry, an analogous argument shows that $g_{e_1}: H_{\bar{\alpha}}^c \rightarrow \mathbb{F}_2$ is also a bent function.
\end{proof}

Alternatively, Case 2 can be resolved without explicitly invoking restricted balancedness by observing the shift in the Walsh spectrum:
\begin{eqnarray*}
\widehat{\mathcal W}_{g_{e_0}}(a, 1) &=& \sum_{x \in H_{\alpha}} (-1)^{g(x) \oplus x \cdot a} - \sum_{x \in H_{\alpha}^c} (-1)^{g(x) \oplus x \cdot a} \\
&=& \sum_{x \in H_{\alpha}} (-1)^{g(x) \oplus x \cdot a \oplus l_{\alpha}(x)} + \sum_{x \in H_{\alpha}^c} (-1)^{g(x) \oplus x \cdot a \oplus l_{\alpha}(x)} \\
&=& \sum_{x \in \mathbb{F}_2^{s-1}} (-1)^{g(x) \oplus x \cdot (a \oplus \alpha)} = \widehat{\mathcal W}_{g}(a \oplus \alpha).
\end{eqnarray*}
The second equality is obtained by the definition of $H_{\alpha}^c$, where $l_{\alpha}(x)=1$ implies $-1 = (-1)^{l_{\alpha}(x)}$. This alternative derivation shows that partitioning the space into affine subspaces does not alter the shape of the Fourier Spectrum (the set of values of the Walsh-Hadamard transform), but merely shifts it.

In the proof of Theorem \ref{BentBalancness}, the specific ordering of the elements in $\mathbb{F}_2^s$ allows us to treat the last coordinate as a constant within each summation. This choice was made without loss of generality by considering $\alpha_s \neq 0$; the results remain invariant regardless of which coordinate is fixed, provided it corresponds to a non-zero component of $\bar{\alpha}$.

\begin{corollary}\label{BentWalsh1} 
Let $g: \mathbb{F}_2^{s-1} \rightarrow \mathbb{F}_2$ be a bent function. For all $a \in \mathbb{F}_2^{s-1}$:
\begin{enumerate}
    \item $\widehat{{\mathcal W}}_{g}(a) = \widehat{{\mathcal W}}_{g_{e_0}}(a, 0) = \widehat{{\mathcal W}}_{g_{e_1}}(a, 0)$.
    \item If $(g \oplus l_a)_{|H_{\alpha}}$ is balanced, then:
    $\widehat{{\mathcal W}}_{g}(a) = 2^{\frac{s-1}{2}} \iff \widehat{{\mathcal W}}_{g_{e_0}}(a, 1) = -2^{\frac{s-1}{2}}$ and $\widehat{{\mathcal W}}_{g_{e_1}}(a, 1) = 2^{\frac{s-1}{2}}$.
    \item If $(g \oplus l_a)_{|H_{\alpha}^c}$ is balanced, then:
    $\widehat{{\mathcal W}}_{g}(a) = -2^{\frac{s-1}{2}} \iff \widehat{{\mathcal W}}_{g_{e_0}}(a, 1) = 2^{\frac{s-1}{2}}$ and $\widehat{{\mathcal W}}_{g_{e_1}}(a, 1) = -2^{\frac{s-1}{2}}$.
\end{enumerate}
\end{corollary}
\qed


\medskip
As anticipated, if a bent function is defined over one of the specified affine spaces, adding a linear function restricted to that same space results in another bent function. Furthermore, similarly to the bent function $g$, the functions $g_{e_0}$ and $g_{e_1}$ are balanced when their domains are restricted. This balancedness is preserved even upon the addition of a linear function. The following corollary details this relationship.

\begin{corollary}\label{balanceextende}
Let $g: \mathbb{F}_2^{s-1} \rightarrow \mathbb{F}_2$ be a bent function and let $\bar{a}=(a, a_s) \in \mathbb{F}_2^s$, where $a \in \mathbb{F}_2^{s-1}$ and $a_s \in \mathbb{F}_2$. Then, for all $\bar{a} \in \mathbb{F}_2^s$, the functions $g_{e_0} \oplus l_{\bar{a}}: H_{\bar{\alpha}} \rightarrow \mathbb{F}_2$ and $g_{e_1} \oplus l_{\bar{a}}: H_{\bar{\alpha}}^c \rightarrow \mathbb{F}_2$ are bent. 

Moreover, if $(g \oplus l_a)_{|H_{\alpha}}$ is balanced, then $g_{e_0} \oplus l_{(a, a_s)}$ is balanced when restricted to $H_{\alpha} \times \{\bar{0}\}$, and $g_{e_1} \oplus l_{(a, a_s)}$ is balanced when restricted to $H_{\alpha} \times \{\bar{1}\}$. Conversely, if $(g \oplus l_a)_{|H_{\alpha}^c}$ is balanced, then $g_{e_0} \oplus l_{(a, a_s)}$ is balanced when restricted to $H_{\alpha}^c \times \{\bar{1}\}$ and $g_{e_1} \oplus l_{(a, a_s)}$ is balanced when restricted to $H_{\alpha}^c \times \{\bar{0}\}$.
\end{corollary}

\begin{proof}
First, we verify that the resulting functions are bent. Given that $|\widehat{{\mathcal W}}_{g_{e_i}}(\bar{b})| = 2^{(s-1)/2}$ for $i \in \{0, 1\}$ and for all $\bar{b} \in \mathbb{F}_2^s$, the linearity of $l_{\bar{a}}$ implies that $g_{e_i} \oplus l_{\bar{a}}$ remains a bent function for all $\bar{a} \in \mathbb{F}_2^s$. Specifically:
\begin{equation*}
\widehat{\mathcal W}_{g_{e_0} \oplus l_{\bar{a}}}(\bar{b}) = \sum_{\bar{x} \in H_{\bar{\alpha}}} (-1)^{g_{e_0}(\bar{x}) \oplus \bar{x} \cdot \bar{a} \oplus \bar{x} \cdot \bar{b}} = \sum_{\bar{x} \in H_{\bar{\alpha}}} (-1)^{g_{e_0}(\bar{x}) \oplus \bar{x} \cdot (\bar{a} \oplus \bar{b})} = \pm 2^{\frac{s-1}{2}}.
\end{equation*}

Regarding the balancedness property, since $g \oplus l_a$ is a bent function, by Theorem \ref{balancedness}, we assume that $g \oplus l_a: \mathbb{F}_2^{s-1} \rightarrow \mathbb{F}_2$ is balanced when restricted to $H_{\alpha}$. Since the value $a_s \in \mathbb{F}_2$ remains constant within each of the domains $H_{\alpha} \times \{\bar{0}\}$ and $H_{\alpha} \times \{\bar{1}\}$, and given that the values of $g$, $g_{e_0}$, and $g_{e_1}$ coincide on these points, it follows that $g_{e_0} \oplus l_{(a, a_s)}$ is balanced on $H_{\alpha} \times \{\bar{0}\}$ and $g_{e_1} \oplus l_{(a, a_s)}$ is balanced on $H_{\alpha} \times \{\bar{1}\}$.

An analogous argument applies if the balancedness restricted to $H_{\alpha}^c$ is considered instead.
\end{proof}

In the preceding corollary, it is crucial to note that we have carefully utilised the corresponding domains for each defined function. Furthermore, the balancedness condition is derived from $g \oplus l_a$ rather than $g$ itself, as the addition of a linear function may alter the specific balanced subset of $g$.

\medskip
To clarify the transitions among shifting dimensions, we move from the general $n$-variable framework to a recursive setup. Here, the base components operate on $\mathbb{F}_2^{s-1}$ and are extended to $\mathbb{F}_2^s$ and $\mathbb{F}_2^{s+1}$. To avoid unnecessary friction regarding the simultaneous use of vectors and hyperplanes in different ambient spaces, a summary of notation is provided in Table \ref{tab:notation}.

\begin{table}[h!]
\centering
\caption{Summary of notation, spaces, and domains for the secondary construction.}
\label{tab:notation}
\begin{tabular}{ll}
\hline
\textbf{Notation}  & \textbf{Description} \\ \hline
$x\in\mathbb{F}_2^{s-1}$ & Vector of dimension $s-1$. \\
$\bar{x}\in \mathbb{F}_2^s$ & Vector of dimension $s$, decomposed as $(x, x_s)$ with $x_s \in \mathbb{F}_2$. \\
$(x_0, \bar{x})\in \mathbb{F}_2^{s+1}$ & Concatenated vector of dimension $s+1$ with $x_0 \in \mathbb{F}_2$. \\
$\alpha \in \mathbb{F}_2^{s-1}$ & Vector defining linear/affine structures in lower dimensions. \\
$\bar{\alpha} \in \mathbb{F}_2^s$ & Vector defining hyperplanes in $\mathbb{F}_2^s$. \\
$\bar{\beta}\in \mathbb{F}_2^{s+1}$ & Vector defined as $(\alpha_0, \alpha, \alpha_s)$ with $\alpha \in \mathbb{F}_2^{s-1}$ and $\alpha_0, \alpha_s \in \mathbb{F}_2$. \\
$H_{\bar{\alpha}} \subset \mathbb{F}_2^s$ & Linear hyperplane in $\mathbb{F}_2^s$ defined by $\bar{\alpha}$ ($\bar{\alpha} \cdot \bar{x} = 0$). \\
$H_{\bar{\alpha}}^c\subset \mathbb{F}_2^s$ & Complementary affine hyperplane in $\mathbb{F}_2^s$ ($\bar{\alpha} \cdot \bar{x} = 1$). \\
$H_{\bar{\beta}}\subset \mathbb{F}_2^{s+1}$ & Linear hyperplane in $\mathbb{F}_2^{s+1}$ defined by $\bar{\beta}$ ($\bar{\beta} \cdot (x_0, \bar{x}) = 0$). \\
$H_{\bar{\beta}}^c \subset \mathbb{F}_2^{s+1}$ & Complementary affine hyperplane in $\mathbb{F}_2^{s+1}$ ($\bar{\beta} \cdot (x_0, \bar{x}) = 1$). \\
$g:\mathbb{F}_2^{s-1} \rightarrow \mathbb{F}_2$ & Bent function of lower dimension. \\
$g_e: \mathbb{F}_2^s \rightarrow \mathbb{F}_2$ & Extended function satisfying $g_e(x, x_s) = g(x)$. \\
$g_{e_0}, g_{e_1}: H_{\bar{\alpha}}, H_{\bar{\alpha}}^c \rightarrow \mathbb{F}_2$ & Restrictions of $g_e$ onto the hyperplanes $H_{\bar{\alpha}}$ and $H_{\bar{\alpha}}^c$. \\ \hline
\end{tabular}
\end{table}

\section{Construction of bent functions via the Maiorana-McFarland secondary construction} 
\label{Secondary Constrution}

From this point forward, let us consider the element $\bar{\beta} = (\alpha_0, \alpha, \alpha_s)$, where $\alpha \in \mathbb{F}_2^{s-1}$ and $\alpha_0, \alpha_s \in \mathbb{F}_2$. We further assume that both $\alpha_0$ and $\alpha_s$ are non-zero. In a manner similar to our previous definitions, we define the set:
\[ H_{\bar{\beta}} := \{ (x_0, \bar{x}) \in \mathbb{F}_2^{s+1} \mid l_{\bar{\beta}}(x_0, \bar{x}) = 0 \}, \]
where $\bar{x} = (x, x_s) \in \mathbb{F}_2^s$, $x \in \mathbb{F}_2^{s-1}$, and $x_0 \in \mathbb{F}_2$.

\medskip
In Theorems \ref{balancedextended} and \ref{UsingMaiorana}, we employ $\phi$ and $g_e$ as previously defined. Recall that $\phi^{-1}(0) = H_{\bar{\alpha}}$ and $\phi^{-1}(1) = H_{\bar{\alpha}}^c$.

We are now prepared to apply the secondary Maiorana-McFarland construction for the particular case where $r=1$. In the following theorem, we construct a bent function over a higher-dimensional domain; furthermore, we demonstrate that the balancedness property is specifically attained over the vector subspace $H_{\bar{\beta}}$.

\begin{theorem}\label{balancedextended}
Let $\bar{\beta} = (\alpha_0, \alpha, \alpha_s)$, with $\alpha \in \mathbb{F}_2^{s-1}$ and $\alpha_0, \alpha_s \in \mathbb{F}_2 \setminus \{0\}$. Let $g: \mathbb{F}_2^{s-1} \rightarrow \mathbb{F}_2$ be a bent function. Then, the function 
\[ f: \mathbb{F}_2^{s+1} \rightarrow \mathbb{F}_2, \quad (x_0, \bar{x}) \mapsto x_0 \phi(\bar{x}) \oplus g_e(\bar{x}), \]
where $x_0 \in \mathbb{F}_2$ and $\bar{x} \in \mathbb{F}_2^s$, is a bent function, and its restriction $f_{|H_{\bar{\beta}}}$ is balanced.
\end{theorem}

\begin{proof}
Since $\phi$ and $g_e$ satisfy the requirements of Theorem \ref{Maiorana}, it follows immediately that $f$ is a bent function.

To examine the balancedness of $f$, let $\bar{x} = (x, x_s)$ with $x \in \mathbb{F}_2^{s-1}$ and $x_s \in \mathbb{F}_2$. The function is defined as $f(x_0, x, x_s) = x_0 \phi(x, x_s) \oplus g_e(x, x_s)$. We consider the following four cases for $g$:
\begin{itemize}
    \item $g_{|H_{\alpha}}$ is balanced and $\widehat{\mathcal{W}}_{g}(\bar{0}) = 2^{(s-1)/2}$.
    \item $g_{|H_{\alpha}}$ is balanced and $\widehat{\mathcal{W}}_{g}(\bar{0}) = -2^{(s-1)/2}$.
    \item $g_{|H_{\alpha}^c}$ is balanced and $\widehat{\mathcal{W}}_{g}(\bar{0}) = 2^{(s-1)/2}$.
    \item $g_{|H_{\alpha}^c}$ is balanced and $\widehat{\mathcal{W}}_{g}(\bar{0}) = -2^{(s-1)/2}$.
\end{itemize}

Assume the case where $g_{|H_{\alpha}}$ is balanced and $\widehat{\mathcal{W}}_{g}(\bar{0}) = 2^{(s-1)/2}$. The elements within the subspace $H_{\bar{\beta}}$ are partitioned into the following four cases:

\textbf{Case 1.} Let $x_0=0$, $x_s=0$, and $x \in H_{\alpha}$. Then, given the balancedness over $H_{\alpha}$ and Remark \ref{DistriWeight}, it follows that:
\[ w_H(f_{|(0, H_{\alpha}, 0)}) = w_H({g_{e_0}}_{|(H_{\alpha}, 0)}) = 2^{s-3}. \]

\textbf{Case 2.} Let $x_0=0$, $x_s=1$, and $x \in H_{\alpha}^c$. In this case, the function corresponds to the restriction of $g_{e_0}$ to the affine shift of the subspace. Following the weight distribution established in Remark \ref{DistriWeight}, we have:
\[ w_H(f_{|(0, H_{\alpha}^c, 1)}) = w_H({g_{e_0}}_{|(H_{\alpha}^c, 1)}) = 2^{s-3} - 2^{(s-3)/2}. \]

\textbf{Case 3.} Let $x_0=1$, $x_s=0$, and $x \in H_{\alpha}^c$. Here, the function is defined by the complement $1 \oplus g_{e_1}$. By applying Remark \ref{DistriWeight} and considering the complementary weight properties, it follows that:
\[ w_H(f_{|(1, H_{\alpha}^c, 0)}) = w_H((1 \oplus g_{e_1})_{|(H_{\alpha}^c, 0)}) = 2^{s-3} + 2^{(s-3)/2}. \]

\textbf{Case 4.} Let $x_0=1$, $x_s=1$, and $x \in H_{\alpha}$. Then, given the balancedness over $H_{\alpha}$ and Remark \ref{DistriWeight}, it follows that:
\[ w_H(f_{|(1, H_{\alpha}, 1)}) = w_H((1 \oplus g_{e_1})_{|(H_{\alpha}, 1)}) = 2^{s-3}. \]

Combining the contributions from these four cases, we obtain the total weight:
\[ \#f_{|H_{\bar{\beta}}}^{-1}(1) = 2^{s-3} + (2^{s-3} - 2^{(s-3)/2}) + (2^{s-3} + 2^{(s-3)/2}) + 2^{s-3} = 4(2^{s-3}) = 2^{s-1}. \]
Since the cardinality of $H_{\bar{\beta}}$ is $2^s$, the function $f_{|H_{\bar{\beta}}}$ is balanced. An analogous result is obtained for the remaining cases.
\end{proof}

It should be noted that the non-zero component of $\alpha$ is not strictly required to be $\alpha_s$. Furthermore, with slight modifications to the theorem, the condition on the element $\bar{\beta} \in \mathbb{F}_2^{s+1}$ can be weakened, requiring only that at least two of its components are non-zero.

\medskip
We have established the balancedness of a bent function restricted to a specific affine space utilising the secondary Maiorana-McFarland construction. Our objective now is to determine the effect on this balance when a linear function is added to the resulting bent function. The balancedness results for a specific affine space do not seem to apply directly to the obtained bent function.   

To prove Theorem \ref{UsingMaiorana}, we begin with the function $g_e \oplus l_{(a, a_s)}$ and examine the various balancedness cases detailed in Corollary \ref{balanceextende}, leveraging Remark \ref{DistriWeight}. Consequently, in the proof of the following theorem, when we refer to the balancedness of the function $g_e \oplus l_{(a, a_s)}$ over a set (for instance, $H_{\alpha} \times \{ \bar{0} \}$), we are specifically addressing the balancedness of the corresponding restricted function (in this case, $g_{e_0} \oplus l_{(a, a_s)}$).

\begin{theorem}\label{UsingMaiorana} 
Let $g: \mathbb{F}_2^{s-1} \rightarrow \mathbb{F}_2$ be a bent function and let 
\[ f: \mathbb{F}_2^{s+1} \rightarrow \mathbb{F}_2, \quad (x_0, \bar{x}) \mapsto x_0 \phi(\bar{x}) \oplus g_e(\bar{x}), \]
where $x_0 \in \mathbb{F}_2$ and $\bar{x} \in \mathbb{F}_2^s$, be a bent function obtained via the secondary Maiorana-McFarland construction. Then, the function $f \oplus l_{(a_0, \bar{a})}$, where $l_{(a_0, \bar{a})}(x_0, \bar{x}) = a_0 x_0 \oplus \bar{a} \cdot \bar{x}$ with $a_0 \in \mathbb{F}_2$ and $\bar{a} = (a_1, \dots, a_s) \in \mathbb{F}_2^s$, satisfies the following conditions:
\begin{enumerate}
    \item If $(a_0=0 \wedge a_s=0) \vee (a_0=1 \wedge a_s=1)$, then $(f \oplus l_{(a_0, \bar{a})})_{|H_{\bar{\beta}}}$ is balanced.
    \item If $(a_0=0 \wedge a_s=1) \vee (a_0=1 \wedge a_s=0)$, then $(f \oplus l_{(a_0, \bar{a})})_{|H_{\bar{\beta}}^c}$ is balanced.
\end{enumerate}
\end{theorem}

\begin{proof}
Let $\bar{a}=(a, a_s)$ and $\bar{x}=(x, x_s)$, where $a, x \in \mathbb{F}_2^{s-1}$ and $a_s, x_s \in \mathbb{F}_2$. There are four possibilities for the pair $(a_0, a_s)$. For each case, we assume the balancedness of the function $g_e \oplus l_{(a, a_s)}$ over each of the following four sets:
\[ H_{\alpha} \times \{ \bar{0} \}, \quad H_{\alpha}^c \times \{ \bar{1} \}, \quad H_{\alpha} \times \{ \bar{1} \}, \quad \text{and} \quad H_{\alpha}^c \times \{ \bar{0} \}. \]
Additionally, we consider the corresponding value of $\widehat{\mathcal{W}}_{g_e \oplus l_{(a, a_s)}}(\bar{0})$, which equals either $2^{(s-1)/2}$ or $-2^{(s-1)/2}$ in the respective complements of the considered affine spaces:
\[ H_{\alpha}^c \times \{ \bar{1} \}, \quad H_{\alpha} \times \{ \bar{0} \}, \quad H_{\alpha}^c \times \{ \bar{0} \}, \quad \text{and} \quad H_{\alpha} \times \{ \bar{1} \}. \]

The function is expressed as:
\[ (f \oplus l_{(a_0, a, a_s)})(x_0, x, x_s) = x_0 \phi(x, x_s) \oplus g_e(x, x_s) \oplus l_{(a, a_s)}(x, x_s) \oplus a_0 x_0. \]

Without loss of generality, let us examine the case where $a_0=0$ and $a_s=1$, assuming $g_e \oplus l_{(a, 1)}$ is balanced in $H_{\alpha} \times \{ \bar{1} \}$ and $\widehat{\mathcal{W}}_{g_e \oplus l_{(a, 1)}}(\bar{0}) = -2^{(s-1)/2}$ in $H_{\alpha}^c \times \{ \bar{0} \}$. The elements $(x_0, x, x_s) \in \mathbb{F}_2^{s+1}$ such that $l_{\bar{\beta}}(x_0, x, x_s) = 1$ are partitioned into the following four cases:

\textbf{Case 1.} Let $x_0=0$, $x_s=0$, and $l_{\alpha}(x)=1$. Since $g_e \oplus l_{(a, 1)}$ is not balanced and $\widehat{\mathcal{W}}_{g_e \oplus l_{(a, 1)}}(\bar{0}) = -2^{(s-1)/2}$ in $H_{\alpha}^c \times \{ \bar{0} \}$, we have:
\[ w_H((f \oplus l_{(0, a, 1)})_{|(0, H_{\alpha}^c, 0)}) = w_H((g_{e_1} \oplus l_{(a, 1)})_{|(H_{\alpha}^c, 0)}) = 2^{s-3} + 2^{(s-3)/2}. \]

\textbf{Case 2.} Let $x_0=0$, $x_s=1$, and $l_{\alpha}(x)=0$. Since $g_e \oplus l_{(a, 1)}$ is balanced in $H_{\alpha} \times \{ \bar{1} \}$, we have:
\[ w_H((f \oplus l_{(0, a, 1)})_{|(0, H_{\alpha}, 1)}) = w_H((g_{e_1} \oplus l_{(a, 1)})_{|(H_{\alpha}, 1)}) = 2^{s-3}. \]

\textbf{Case 3.} Let $x_0=1$, $x_s=0$, and $l_{\alpha}(x)=0$. Since $g_e \oplus l_{(a, 1)}$ is balanced in $H_{\alpha} \times \{ \bar{0} \}$, we have:
\[ w_H((f \oplus l_{(0, a, 1)})_{|(1, H_{\alpha}, 0)}) = w_H((g_{e_0} \oplus l_{(a, 1)})_{|(H_{\alpha}, 0)}) = 2^{s-3}. \]

\textbf{Case 4.} Let $x_0=1$, $x_s=1$, and $l_{\alpha}(x)=1$. Since $g_e \oplus l_{(a, 1)}$ is not balanced and $\widehat{\mathcal{W}}_{\text{g}_e \oplus l_{(a, 1)}}(\bar{0}) = 2^{(s-1)/2}$ in $H_{\alpha}^c \times \{ \bar{1} \}$, we have:
\[ w_H((f \oplus l_{(0, a, 1)})_{|(1, H_{\alpha}^c, 1)}) = w_H((g_{e_0} \oplus l_{(a, 1)})_{|(H_{\alpha}^c, 1)}) = 2^{s-3} - 2^{(s-3)/2}. \]

Combining these four cases, we obtain:
\[ \#(f \oplus l_{(a_0, a, a_s)})_{|H_{\bar{\beta}}^c}^{-1}(1) = 2^{s-1}. \]
Therefore, $f_{|H_{\bar{\beta}}^c}$ is balanced. Similar results are obtained for the remaining cases.
\end{proof}

It is evident that the restricted balancedness of the bent function generated via the secondary Maiorana-McFarland construction, under the addition of the linear function $l_{(a_0, a, a_s)}$, depends solely on the parameters $a_0$ and $a_s$. Furthermore, these elements explicitly determine the specific affine spaces where the balancedness property is satisfied.

\medskip
By analysing all the cases related to the proofs of Theorems \ref{balancedextended} and \ref{UsingMaiorana}, including those instances not explicitly resolved due to their similarity, we derive Algorithms \ref{tb1} and \ref{tb2}. These algorithms are specifically designed for the case where $\alpha = \bar{1}$ at each iteration. In this framework, $\alpha$ assumes the appropriate dimension for the vector space at each step, defining the affine spaces as the sets of vectors with even and odd Hamming weights, respectively. To generalise these results for any value of $\alpha$, the conditions $\alpha_0 \neq 0$ and $\alpha_s \neq 0$ must be maintained at each iteration, following the notation established previously.

In these algorithms, we define the sets $\mathcal{A}$ and $\mathcal{B}$, where $\mathcal{A}$ contains vectors of length $o$ with even Hamming weights, and $\mathcal{B}$ consists of vectors of length $o$ with odd Hamming weights. Furthermore, we denote $l_a(x) := a_1x_1 \oplus \dots \oplus a_rx_r$, for $a, x \in \mathbb{F}_2^r$ and any positive integer $r$. In the first algorithm, we derive bent functions for any even dimension in the domain that exceeds the dimension of the given bent function. These newly constructed bent functions are explicitly balanced when the domain is restricted to the set of vectors with even Hamming weight. In the second algorithm, we examine the sum of a linear function and identify the specific conditions under which its balancedness is preserved. This restricted balancedness is uniquely determined by the specific linear components, allowing for the precise identification of the affine spaces where the property holds.

\medskip
In Algorithm 1, the implementation of the assignment steps within the loop is directly derived from the combinatorial case-by-case construction proven in Theorem \ref{balancedextended}. The conditional (if-else) structures process the binary combinations of the external coordinates $(x'^n, x''^n)$ (corresponding to $x_0$ and $x_s$ in the proof) by considering whether the Hamming weight $w_H(x^n)$ is even or odd, which identifies membership in $H_{\alpha}$ or $H_{\alpha}^c$. Note also that the function $1 \oplus g_n(x^n)$ in the algorithm is analogue to Cases 3 and 4 of the Theorem, whilst the function $g_n(x^n)$ is analogue to Cases 1 and 2. Consequent to the assertion of Theorem \ref{balancedextended}, the algorithm consistently preserves the weights $2^{s-3}$ and $2^{s-3} \pm 2^{(s-3)/2}$ required to reach the total balanced weight of $2^{s-1}$. The algorithm iterates over the dimension parameter $n$, starting from an initial even dimension $s-1$ and continuing until the target dimension $o$ is attained. At each step of the while loop, the dimension of the constructed bent function increases exactly by 2 because the secondary Maiorana-McFarland framework extends the domain by incorporating two external bits, $x'^n$ and $x''^n$, mapping the input function $g_n$ of dimension $n$ to an expanded function $g_{n+2}$ of dimension $n+2$. The input assumptions that guarantee the deterministic termination of this process are that the target dimension $o$ is a finite integer such that $o > s-1$, and that the seed function $g_{s-1}$ is a valid bent function. Since $n$ increases strictly by a constant step of 2 at each iteration, the condition $n < o$ becomes false after exactly $\frac{o - (s-1)}{2}$ steps, ensuring that the algorithm halts.

\begin{algorithm}
    \caption{{\tt Maiorana-McFarland secondary construction} $r=1$}
    \label{tb1}
    \begin{minipage}{\hsize}
    \begin{algorithmic}[1]
        \Require $s-1\geq 2$ even, $g_{s-1}: {\mathbb F}_2^{s-1}\rightarrow {\mathbb F}_2$ be a bent function, $x'^{s-1}, x''^{s-1}\in{\mathbb F}_2, x^{s-1}\in {\mathbb F}_2^{s-1}$, and the target dimension $o$ (even integer with $o > s-1$).
        \Ensure $g_{o}$ a bent function, ${g_{o}}_{|{\mathcal A}}$ balanced
        \State $n:=s-1;$
        \While{$n < o$}
            \For{$x'^n,x''^n$ from $0$ to $1$}
                \If{($w_H(x^n)$ is even, $x'^n=1$, $x''^n=1$) \textbf{or} ($w_H(x^n)$ is odd, $x'^n=1$, $x''^n=0$)}
                    \State $g_{n+2}(x'^n,x^n,x''^n)=1\oplus g_{n}(x^n);$
                \Else
                    \State $g_{n+2}(x'^n,x^n,x''^n)=g_{n}(x^n);$
                \EndIf
            \EndFor
            \State $x^n=(x'^n,x^n,x''^n);$ ~ $n:= n+2;$        
        \EndWhile
        \State \Return $g_{o};$
    \end{algorithmic}
    \end{minipage}
\end{algorithm}

\medskip
Algorithm 2 formalises the iterative construction and restricted balancedness of bent functions under linear perturbations derived from Theorem \ref{UsingMaiorana}. The algorithm iterates over the dimension parameter $n$, starting from an initial even dimension $s-1$ until the target dimension $o$ is reached. At each iteration of the while loop, the dimension increases exactly by 2 because the construction incorporates two external bits, $x'^n$ and $x''^n$, and two external linear shift bits, $a'^n$ and $a''^n$, mapping the linear perturbation $g_n \oplus l_{a^n}$ over $\mathbb{F}_2^n$ to an expanded function $g_{n+2}\oplus l_{(a'^n,a^n,a''^n)}$. The code branches into four blocks depending on the cases of $(a'^n, a''^n)$ to evaluate the specific weight distributions that guarantee restricted balancedness over the corresponding affine spaces $\mathcal{A}$ or $\mathcal{B}$. The four primary blocks of this algorithm are directly derived from the four case-by-case partitions analysed in the proof of Theorem \ref{balancedextended}. The input assumptions ensuring the deterministic termination of this algorithm are that the target dimension $o$ is a finite integer satisfying $o > s-1$ and that the initial function $g_{s-1}$ is a valid bent function. Since the dimension $n$ increases by a constant step of 2 at the end of each valid branch, the condition $n < o$ becomes false after exactly $\frac{o - (s-1)}{2}$ steps, guaranteeing that the execution halts.

\begin{algorithm}
    \caption{{\tt Maiorana-McFarland secondary construction} $\oplus linear ~r=1$}
    \label{tb2}
    \begin{minipage}{\hsize}
    \begin{algorithmic}[1]
        \Require $s-1\geq 2$ even, $g_{s-1}: {\mathbb F}_2^{s-1}\rightarrow {\mathbb F}_2$ be a bent function, $l_{a^{s-1}}$ a linear function on ${\mathbb F}_2^{s-1},$ $x'^{s-1}, a'^{s-1}, x''^{s-1}, a''^{s-1} \in{\mathbb F}_2, x^{s-1}, a^{s-1}\in {\mathbb F}_2^{s-1}$, and the target dimension $o$ (even integer with $o > s-1$).
        \Ensure $g_{o}\oplus l_{(a'^{o-2},a^{o-2},a''^{o-2})}$ a bent function. If $(a'^{o-2}=0 \wedge a''^{o-2}=0) \vee (a'^{o-2}=1 \wedge a''^{o-2}=1)$, $({g_{o}} \oplus l_{(a'^{o-2},a^{o-2},a''^{o-2})})_{|{\mathcal A}}$ balanced. If $(a'^{o-2}=0 \wedge a''^{o-2}=1) \vee (a'^{o-2}=1 \wedge a''^{o-2}=0)$, $({g_{o}} \oplus l_{(a'^{o-2},a^{o-2},a''^{o-2})})_{|{\mathcal B}}$ balanced
        \State $n:=s-1;$
        \While{$n < o$}
            \If{$a'^n = 0$ \textbf{and} $a''^n=0$}
                \For{$x'^n, x''^n$ from $0$ to $1$}
                    \If{($w_H(x^n)$ is even, $x'^n=1$, $x''^n=1$) or ($w_H(x^n)$ is odd, $x'^n=1$, $x''^n=0$)}
                        \State $(g_{n+2}\oplus l_{(0,a^n,0)})(x'^n,x^n,x''^n)=1\oplus (g_{n}\oplus l_{a^n})(x^n);$
                    \Else
                        \State $(g_{n+2}\oplus l_{(0,a^n,0)})(x'^n,x^n,x''^n)=(g_{n}\oplus l_{a^n})(x^n);$
                    \EndIf                    
                \EndFor
                \State $x^n=(x'^n,x^n,x''^n);$ ~$a^n=(a'^n,a^n,a''^n);$ ~$n:= n+2;$
            \EndIf
            \If{$a'^n = 0$ \textbf{and} $a''^n=1$}
                \For{$x'^n, x''^n$ from $0$ to $1$}
                    \If{($w_H(x^n)$ is even, $x'^n=0$, $x''^n=1$) or \{($w_H(x^n)$ is odd, [($x'^n=0$, $x''^n=1$) or ($x'^n=1$, $x''^n=0$) or ($x'^n=1$, $x''^n=1$)]\}}
                        \State $(g_{n+2}\oplus l_{(0,a^n,1)})(x'^n,x^n,x''^n)=1\oplus (g_{n}\oplus l_{a^n})(x^n);$
                    \Else
                        \State $(g_{n+2}\oplus l_{(0,a^n,1)})(x'^n,x^n,x''^n)=(g_{n}\oplus l_{a^n})(x^n);$
                    \EndIf
                \EndFor
                \State $x^n=(x'^n,x^n,x''^n);$ ~$a^n=(a'^n,a^n,a''^n);$ ~$n:= n+2;$             \EndIf
            \If{$a'^n = 1$ \textbf{and} $a''^n=0$}
                \For{$x'^n, x''^n$ from $0$ to $1$}
                    \If{($w_H(x^n)$ is even, $x'^n=1$, $x''^n=0$) or ($w_H(x^n)$ is odd, $x'^n=1$, $x''^n=1$)}
                        \State $(g_{n+2}\oplus l_{(1,a^n,0)})(x'^n,x^n,x''^n)=1\oplus (g_{n}\oplus l_{a^n})(x^n);$
                    \Else
                        \State $(g_{n+2}\oplus l_{(1,a^n,0)})(x'^n,x^n,x''^n)=(g_{n}\oplus l_{a^n})(x^n);$
                    \EndIf
                \EndFor
                \State $x^n=(x'^n,x^n,x''^n);$ ~$a^n=(a'^n,a^n,a''^n);$ ~$n:= n+2;$ 
            \EndIf
            \If{$a'^n = 1$ \textbf{and} $a''^n=1$}
                \For{$x'^n, x''^n$ from $0$ to $1$}
                    \If{\{$w_H(x^n)$ is even, [($x'^n=0$, $x''^n=1$) or ($x'^n=1$, $x''^n=0$) or ($x'^n=1$, $x''^n=1$)]\} or ($w_H(x^n)$ is odd, $x'^n=0$, $x''^n=1$)}
                        \State $(g_{n+2}\oplus l_{(1,a^n,1)})(x'^n,x^n,x''^n)=1\oplus (g_{n}\oplus l_{a^n})(x^n);$
                    \Else
                        \State $(g_{n+2}\oplus l_{(1,a^n,1)})(x'^n,x^n,x''^n)=(g_{n}\oplus l_{a^n})(x^n);$
                    \EndIf
                \EndFor
                \State $x^n=(x'^n,x^n,x''^n);$ ~$a^n=(a'^n,a^n,a''^n);$ ~$n:= n+2;$  
            \EndIf
        \EndWhile
        \State \Return $g_{o}\oplus l_{(a'^{o-2},a^{o-2},a''^{o-2})};$
    \end{algorithmic}
    \end{minipage}
\end{algorithm}

\begin{example}
To illustrate Algorithm 1 and Algorithm 2, we present an example expanding a seed bent function from dimension $n=2$ to $n=4$. Subsequently, its balancedness under a linear perturbation is verified.

{\bf Step 1: Initialisation with a Seed Function ($n=2$)}

Let $s-1 = 2$. The initial vector space is $\mathbb{F}_2^2$ with elements denoted as $x = (x_1, x_2)$. We define the standard seed bent function $g_2: \mathbb{F}_2^2 \rightarrow \mathbb{F}_2$ as follows:
\begin{equation}
    g_2(x_1, x_2) = x_1 x_2.
\end{equation}

{\bf Step 2: Iterative Expansion to $n=4$ via Algorithm 1}

We set the target dimension to $o = 4$. Algorithm 1 enters the while loop with $n = 2$. For each vector, two external bits are incorporated (a left bit $x'$ and a right bit $x''$) generating the expanded vector $(x', x_1, x_2, x'') \in \mathbb{F}_2^4$. 

The conditional assignment maps the input function according to the following criteria:
\begin{equation}
    g_4(x', x_1, x_2, x'') = 
    \begin{cases} 
    1 \oplus g_2(x_1, x_2) & \text{if } (w_H(x) \text{ is even } \wedge x'=1 \wedge x''=1) \\
                           & \text{or } (w_H(x) \text{ is odd } \wedge x'=1 \wedge x''=0), \\
    g_2(x_1, x_2)          & \text{otherwise.}
    \end{cases}
\end{equation}

By exhaustively executing these loop iterations across all 16 cases, we obtain the algebraic normal form for $g_4$:
\begin{equation}
    g_4(x_1, x_2, x_3, x_4) = x_2 x_3 \oplus x_1 x_4 \oplus x_1 x_2 \oplus x_1 x_3,
\end{equation}
where variables are re-indexed for standard notation ($x' \to x_1$, $x_1 \to x_2$, $x_2 \to x_3$, $x'' \to x_4$).

{\bf Step 3: Truth Table}

Table \ref{truth_table_g4} presents the outputs for the bent function $g_4$, $g_4 \oplus l_{(0,0,0,1)}$, and $g_4 \oplus l_{(0,1,0,1)}.$

\begin{table}[h]
\centering
\caption{Complete truth table for the extended bent function $g_4(x_1, x_2, x_3, x_4)$ and sums of linear functions.}
\label{truth_table_g4}
\begin{tabular}{ccccccc}
\hline
$x_1$ & $x_2$ & $x_3$ & $x_4$ & \multicolumn{1}{|c}{$g_4(x_1, x_2, x_3, x_4)$} & \multicolumn{1}{|c}{$g_4(x_1, x_2, x_3, x_4)\oplus x_4$} & \multicolumn{1}{|c}{$g_4(x_1, x_2, x_3, x_4) \oplus x_2\oplus x_4$} \\ \hline
0    & 0     & 0     & 0     & 0                        & 0 & 0 \\
1    & 1     & 0     & 0     & 1                        & 1 & 0 \\
0    & 1     & 1     & 0     & 1                        & 1 & 0 \\
1    & 0     & 1     & 0     & 1                        & 1 & 1 \\
0    & 0     & 1     & 1     & 0                        & 1 & 1 \\
1    & 1     & 1     & 1     & 0                        & 1 & 0 \\
0    & 1     & 0     & 1     & 0                        & 1 & 0 \\
1    & 0     & 0     & 1     & 1                        & 0 & 0 \\ \hline
0    & 0     & 0     & 1     & 0                        & 1 & 1 \\
1    & 1     & 0     & 1     & 0                        & 1 & 0 \\
0    & 1     & 1     & 1     & 1                        & 0 & 1 \\
1    & 0     & 1     & 1     & 0                        & 1 & 1 \\
0    & 0     & 1     & 0     & 0                        & 0 & 0 \\
1    & 1     & 1     & 0     & 1                        & 1 & 0 \\
0    & 1     & 0     & 0     & 0                        & 0 & 1 \\
1    & 0     & 0     & 0     & 0                        & 0 & 0 \\ \hline
\end{tabular}
\end{table}

The balancedness of the bent functions can be observed in the table. Furthermore, as indicated by Algorithm 2, if $a_1=0$ and $a_4=1$, then balancedness is achieved on words of odd Hamming weight, independently of the values of $a_2$ and $a_3$.   
\end{example}

\section{Conclusions}
\label{Conclusions}

In this work, we have presented a study of the cryptographic properties of bent functions under affine restrictions. A primary contribution of this research is the development of an alternative proof for both the balancedness and the bent property of these functions when their domain is restricted to specific affine spaces, such as those defined by even or odd Hamming weights. 

Furthermore, we have introduced original results by applying the secondary Maiorana-McFarland construction. Our analysis in Theorem \ref{balancedextended} reveals an important characteristic: the resulting bent functions exhibit balancedness exclusively within the defined affine subspace $H_{\bar{\beta}}$. This precise localisation of the balancedness property is significant for understanding how secondary constructions distribute their weight across high-dimensional domains.

This study also conducted an analysis of these constructions and the addition of linear functions to them. Our findings are partially reflected in two algorithms derived from our theoretical framework; these algorithms provide a systematic method for constructing bent functions with strictly controlled balancedness properties.

As suggested in \cite{MaitraBalance}, starting from bent functions remains a viable approach to obtaining Boolean functions with high non-linearity and total balancedness. Since our current approach successfully preserves balancedness over a precisely identified half of the domain, future research will focus on extending this property to the entire domain while minimising the loss of non-linearity, aiming to bridge the gap between bent and highly non-linear balanced functions.

\begin{acknowledgement}
J.C. Ku-Cauich, S. Mandujano-Velázquez, and V. Bautista-Ancona wish to honour the memory of Dr. Javier Díaz Vargas, whose significant contributions to this research were fundamental to its development.
\end{acknowledgement}

\bibliographystyle{spmpsci}

\bibliography{Bent functions construction using extended Maiorana}

\end{document}